\documentstyle[12pt]{article}

\begin{document}

\newtheorem{theorem}{Theorem}[section]
\newtheorem{corollary}[theorem]{Corollary}
\newtheorem{definition}[theorem]{Definition}
\newtheorem{conjecture}[theorem]{Conjecture}
\newtheorem{question}[theorem]{Question}
\newtheorem{lemma}[theorem]{Lemma}
\newtheorem{proposition}[theorem]{Proposition}
\newtheorem{example}[theorem]{Example}
\newtheorem{problem}[theorem]{Problem}
\newtheorem{remark}[theorem]{Remark}
\newenvironment{proof}{\noindent {\bf
Proof.}}{\rule{3mm}{3mm}\par\medskip}
\newcommand{\pp}{{\it p.}}
\newcommand{\de}{\em}

\newcommand{\JEC}{{\it Europ. J. Combinatorics},  }
\newcommand{\JCTB}{{\it J. Combin. Theory Ser. B.}, }
\newcommand{\JCT}{{\it J. Combin. Theory}, }
\newcommand{\JGT}{{\it J. Graph Theory}, }
\newcommand{\ComHung}{{\it Combinatorica}, }
\newcommand{\DM}{{\it Discrete Math.}, }
\newcommand{\ARS}{{\it Ars Combin.}, }
\newcommand{\SIAMDM}{{\it SIAM J. Discrete Math.}, }
\newcommand{\SIAMADM}{{\it SIAM J. Algebraic Discrete Methods}, }
\newcommand{\SIAMC}{{\it SIAM J. Comput.}, }
\newcommand{\ConAMS}{{\it Contemp. Math. AMS}, }
\newcommand{\TransAMS}{{\it Trans. Amer. Math. Soc.}, }
\newcommand{\AnDM}{{\it Ann. Discrete Math.}, }
\newcommand{\NBS}{{\it J. Res. Nat. Bur. Standards} {\rm B}, }
\newcommand{\ConNum}{{\it Congr. Numer.}, }
\newcommand{\CJM}{{\it Canad. J. Math.}, }
\newcommand{\JLMS}{{\it J. London Math. Soc.}, }
\newcommand{\PLMS}{{\it Proc. London Math. Soc.}, }
\newcommand{\PAMS}{{\it Proc. Amer. Math. Soc.}, }
\newcommand{\JCMCC}{{\it J. Combin. Math. Combin. Comput.}, }
\newcommand{\GC}{{\it Graphs Combin.}, }
\thispagestyle{empty}
\title{The Second Zagreb Indices of Graphs with Given Degree Sequences
\thanks{This work is supported by National Natural Science
Foundation of China (No.11271256), Innovation Program of Shanghai Municipal Education Commission (No.14ZZ016) and Specialized Research Fund for the Doctoral Program of Higher Education (No.20130073110075).\newline \indent
$^{\dagger}$Correspondent author: Xiao-Dong Zhang (Email: xiaodong@sjtu.edu.cn)
}}
\author{  Wei-Gang Yuan, Xiao-Dong Zhang$^{\dagger}$\\
{\small Department of Mathematics, and MOE-LSC}\\
{\small Shanghai Jiao Tong University} \\
{\small  800 Dongchuan Road, Shanghai, 200240, P.R. China}\\
}
\maketitle

 \begin{minipage}{5in}
 \begin{center}
 Abstract
 \end{center}
 The second Zagreb index of a graph G is denoted by $M_2(G)=\sum_{uv\in E(G)}d(u)d(v)$.
 In this paper,  we investigate properties of  the extremal graphs with the maximum second Zagreb indices with given graphic sequences, in particular graphic bicyclic sequences. Moreover, we obtain  the relations of the second Zagreb indices among the extremal graphs with different degree sequences.
 \end{minipage}

\vskip 0.5cm

\section{Introduction}

Throughout this paper, $G=(V,E)$ is a simple undirected graph with vertex set $V$ and edge set $E$.
The distance between two vertices $u$ and $v$ which is denoted by $d(u,v)$ is the length of the shortest path that connects $u$ and $v$.
For a vertex $v\in V$, $N(v)$ denotes the neighbor set of $v$ and $d(v)=|N(v)|$ denotes  the degree of $v$. A vertex whose degree is one is called {\it leaf.} Moreover, $(d(v_1), \cdots, d(v_n))$ is called {\it degree sequence} of $G$.
 A nonnegative non-increased integer sequence $\pi=(d_1, d_2, \ldots , d_n)$ is called the {\it graphic sequence} if there exists a  simple graph $G$ such that its degree sequence is exactly $\pi$.  For convenience, we use $d^{(k)}$ to denote the $k$ same degrees  $d$ in $\pi$.
For example, $\pi=(4,4,2,2,1,1)$ is denoted by $(4^{(2)},2^{(2)},1^{(2)})$.
Let $\pi$ be a given graphic  sequence.  Let
$$\Gamma(\pi)=\{G |\ G{\mbox{ is a connected graph with degree sequences}~\pi}\}.$$
Without loss of generality, assume $d(v_i)=d_i$, for $1\le i\le n$, $v_i\in G\in \Gamma(\pi)$.

The {\it second Zagreb index} \cite{Balaban1983} of a graph $G$ is definted by:
\begin{equation}
M_2(G)=\sum_{uv\in E}d(u)d(v).
\end{equation}
For a given graphic sequence $\pi$, let $$M_2(\pi)=max\{M_2(G): G\in \Gamma(\pi)\}.$$
A simple connected graph $G$ is called an {\it optimal graph} in $\Gamma(\pi)$ if $G\in\Gamma(\pi)$ and $M_2(G)=M_2(\pi)$.

The second Zagreb index, whose origin may be dated back to \cite{Gutman2004} and \cite{Nikolic2003},  plays an important role in total $\pi-$electron energy on molecular structure in chemical graph theory.  There are two excellent surveys (\cite{Gutman2004},\cite{Nikolic2003}) on the Zagreb index, which summarize main properties and characterization of the topological index.  Das et al. \cite{das2014} investigated the connections between the  Zagreb index and the Wiener index.  Estes and Wei \cite{estes2014} presented the sharp upper and lower bounds for the Zagreb indices of $k-$tree.  For more information, the readers are referred to \cite{Balaban1983}, \cite{Gutman2004}, \cite{Gutman1975}, \cite{Kier1976}, \cite{Kier1986},
\cite{Nikolic2003}, \cite{Todeschini2000} and references therein.

Recently, Liu and Liu \cite{Liu2012} characterized the all optimal trees in the set of trees with a given tree sequence.
Further, they \cite{Liu2014} investigate some optimal unicycle graphs
in the set of unicycle graphs with a given unicyclic graphic sequence. In this paper, we study  properties of the optimal graphs in the set of all connected graphs with a given graphic sequence $\pi$ that satisfies some conditions, which generalize the main results in \cite{Liu2012} and \cite{Liu2014}.  In addition,  we present some optimal bicyclic graphs in the set of all bicyclic graphs with a given bicyclic graphic  sequence and some relations of the maximum values of the second Zagreb indices with different bicyclic graphic sequences.
The rest of this paper is organized as follows. In Section 2, some notations and the main results of this paper are presented .  In Sections 3, 4 and 5, the proofs of the main results  are presented, respectively.

\section{Preliminary and Main Results}
In order to present the main results of this paper, we introduce some more notations.
Assume $G$ is a rooted graph with root  $v_1$. Let $h(v)$ be the distance between $v$ and $v_1$ and $H_i(G)$ be  the  set of vertices with distance $i$ from vertex $v_1$.

\begin{definition}\cite{zhang2009}
Let $G=(V,E)$ be a graph of root $v_1$. A well-ordering $\prec$ of the vertices is called breadth-first search
ordering with non-increasing degrees (BFS-ordering for short) if the following holds for all vertices $u,v\in V$ :

(1) $u\prec v$ implies $h(u)\le h(v)$;

(2) $u\prec v$ implies $d(u)\ge d(v)$;

(3) if there are two edges $uu_1\in E(G)$ and $vv_1\in E(G)$ such that $u\prec v$, $h(u)=h(u_1)+1$ and $h(v)=h(v_1)+1$,
then $u_1\prec v_1$. 
\end{definition}

For a graphic sequence $\pi=(d_1,d_2,\ldots ,d_n)$ with $\sum_{i=1}^n d_i=2(n+c)$, $d_1\geq d_2\geq c+2$, c is an integer and $c\geq -1$.
We may construct a graph $G_M^*(\pi)$ by following steps.
Select $v_1$ as the root vertex and begin with $v_1$ of the zeroth layer.
Select the vertices $v_2,v_3,v_4,\ldots,v_{d_1+1}$ as the first layer such that $N(v_1)=\{v_2,v_3,v_4,\ldots,v_{d_1+1}\}$;
then, append $d_2-1$ vertices to $v_2$, $d_3-2$ vertices to $v_3$, $\cdots$, $d_{c+3}-2$ vertices to $v_{c+3}$
such that $N(v_2)=\{v_1,v_3,\ldots,v_{c+3},v_{d_1+2},v_{d_1+3},\ldots,v_{d_1+d_2-c-1}\}$,
$N(v_3)=\{v_1,v_2,v_{d_1+d_2-c},\ldots,v_{d_1+d_2+d_3-c-3}\}$, 
$\cdots$, $N(v_{c+3})=\{v_1,v_2,v_{(\sum_{i=1}^{c+2}d_i)-3c},\ldots,\\v_{(\sum_{i=1}^{c+3}d_i)-3c-3}\}$.
After that, append $d_{c+4}-1$ vertices to $v_{c+4}$ such that
$N(v_{c+4})=\{v_1, v_{(\sum_{i=1}^{c+3}d_i)-3c-2},\ldots,v_{(\sum_{i=1}^{c+4}d_i)-3c-4}\}$; $\cdots$ .
Note that $v_1v_2v_3$, $\ldots$, $v_1v_2v_{c_3}$ form $c+1$ triangles in $G_M^*(\pi)$.
Obviously, $G_M^*({\pi})$ is a BFS-ordering graph. In particular,  if $c=1,$  the graph $G_M^*({\pi})$ is denoted by $B_M^*(\pi)$.

The first main result in this paper can be stated as follows.
\begin{theorem}\label{general}
Let $\pi=(d_1,d_2,\ldots ,d_n)$ be a  graphic sequence. If it satisfies the following condition:

$(i)$\quad$\sum_{i=1}^n d_i=2(n+c)$, c is an integer and $c\geq -1$;

$(ii)$\quad $d_1\geq d_2\geq c+2$;

$(iii)$\quad$d_3\ge d_4=d_5=\cdots=d_{c+3}$, for $c\ge 0$;

$(iv)$\quad$d_n=1$;

then $G_M^*({\pi})$ is an optimal graph in $\Gamma(\pi)$. In other words, for any graph $G\in \Gamma(\pi)$, $M_2(G)\le M_2(G_M^*(\pi))$.
\end{theorem}
\begin{remark}  If $\pi$ is a tree degree sequence, then there exists only one tree with degree $\pi$ having a BFS order (for example, see \cite{Zhang2007}). Hence it follows from Theorem~\ref{general} that the main results in \cite{Liu2012} and \cite{Liu2014} hold for $c=-1$  and $c=0$, respectively.
\begin{corollary}\label{tree}(\cite{Liu2012})
Let $\pi$ be a  tree degree sequence. The BFS-tree in $\Gamma(\pi)$ reaches the maximum second Zagreb index.
\end{corollary}

\begin{corollary}\label{unicyclic}(\cite{Liu2014})
Let $\pi=(d_1, \cdots, d_n)$ be  a unicycle graphic sequence  with $d_n=1$. Then there exists an optimal graph $G\in \Gamma(\pi)$
which has a BFS-ordering $\{v_1, \cdots, v_n\}$ with a triangle $v_1v_2v_3$.
\end{corollary}

Moreover, condition $(iii)$ in Theorem \ref{general} can not be deleted. For example, let $\pi=(4,4,3,3,2,1,1)$ which doesn't satisfy condition$(iii)$.
In Figure 1, $G$ is produced by the method in Theorem \ref{general} and $G'$ is not
isomorphic to $G$. It's easy to see that $M_2(G')=M_2(G)+1$.
\end{remark}

\begin{picture}(60,60)
\put(25,52){$v_{1}$}
\put(25,50){\circle{0.5}}
\put(25,17){$v_{2}$}
\put(25,20){\circle{0.5}}
\put(30,37){$v_{4}$}
\put(25,50){\line(0,-1){30}}\put(25,50){\line(1,-3){5}}\put(30,35){\circle{0.5}}
\put(25,20){\line(1,3){5}}
\put(16,37){$v_{3}$}
\put(20,35){\circle{0.5}}
\put(25,50){\line(-1,-3){5}}\put(25,20){\line(-1,3){5}}
\put(56,37){$v_{5}$}
\put(55,35){\circle{0.5}}
\put(25,50){\line(2,-1){30}}
\put(25,20){\line(2,1){30}}
\put(6,35){$v_{6}$}
\put(10,35){\circle{0.5}}
\put(20,35){\line(-1,0){10}}
\put(40,37){$v_{7}$}
\put(40,35){\circle{0.5}}
\put(30,35){\line(1,0){10}}
\put(30,10){$G$}

\put(105,52){$v_{1}$}
\put(105,17){$v_{2}$}
\put(105,50){\circle{0.5}}\put(105,50){\line(0,-1){30}}\put(105,50){\line(1,-1){15}}\put(120,35){\circle{0.5}}
\put(86,35){$v_{3}$}
\put(90,35){\circle{0.5}}\put(90,35){\line(1,1){15}}\put(90,35){\line(1,-1){15}}
\put(120,37){$v_{4}$}
\put(105,20){\circle{0.5}}\put(105,20){\line(1,1){15}}
\put(90,35){\line(1,0){30}}
\put(120,52){$v_{5}$}
\put(120,50){\circle{0.5}}
\put(105,50){\line(1,0){15}}
\put(120,22){$v_{6}$}
\put(120,20){\circle{0.5}}
\put(105,20){\line(1,0){15}}
\put(135,52){$v_{7}$}
\put(135,50){\circle{0.5}}
\put(120,50){\line(1,0){15}}
\put(108,10){$G'$}
\put(60,5){\bf Figure 1}
\end{picture}

In order to present the results of bicyclic graphs with given bicyclic graphic sequences, we introduce some more notations.

A {\it bicyclic graph} is a connected graph with $n\geq 4$ vertices and $n+1$ edges.
Let $\pi=(d_1, \cdots, d_n)$ be a graphic  sequence.
If $\pi$ is a degree sequence of some bicyclic graphs, $\pi$ is called a {\it bicyclic graphic}. For a given bicyclic graphic  sequence $\pi$, let
$${\mathcal{B}}_{\pi}=\{ G|\ G\  {\mbox{ is bicyclic graph with degree sequences}~\pi}\}$$

If $\pi$ is  a bicyclic graphic sequence, then $\Sigma_{i=1}^n d_i=2n+2$.
Denote by $B(p,q)$  a bicyclic graph of order $n$ obtained from two vertex-disjoint cycles $C_p$ and $C_q$ by identifying vertices $u$ of $C_p$ and $v$ of $C_q$ with $p+q-1=n$. Denote by $B(p,r,q)$  a bicyclic graph of order $n$  obtained from two vertex-disjoint cycles $C_p$ and $C_q$ by joining vertices $u$ of $C_p$ and $v$ of $C_q$
by a new path $uu_1u_2\cdots u_{r-1}v$ with length $r(r\geq1)$ with $p+q+r-1=n.$
 Denote by $B(P_k,P_l,P_m)~(1\leq m\leq min\{k,l\})$ a bicyclic graph of order $n$ obtained from three pairwise internal disjoint paths $xv_1v_2\cdots v_{k-1}y$, $xu_1u_2\cdots u_{l-1}y$ and $xw_1w_2\cdots w_{m-1}y$  with $k+l+m-1=n$.
 Denote by $B(p,q;p_1,p_2,\ldots,p_s)$ a bicyclic graph of order $n$ obtained from $B(p,q)$ appending $s$ paths on the common vertex of the two cycles, where
 $p+q+p_1+\cdots+p_s-1=n$, $s$ is the number of leaves and $p_1,p_2,\ldots,p_s$ denote the lengths of the $s$ paths.


The results of bicyclic graphic sequences can be stated as follows.
\begin{theorem}\label{main} Let $\pi=(d_1, \cdots, d_n)$ be a bicyclic graphic  sequence. Denote by $s$ the number of leaves in the graph of ${\mathcal{B}}_{\pi}$.

(1). If $d_n=2$ and $d_2\ge 3$, then the optimal graphs in the set ${\mathcal{B}}_{\pi}$ are $B(p,1,q)$ or $B(P_k,P_l,P_1)$ with $p+q=n$ and $k+l=n$.
In other words, for any $G\in {\mathcal{B}}_{\pi}$,  $M_2(G)\le 4n+17$  with equality if and only if $G$ is $B(p,1,q)$ or $B(P_k,P_l,P_1)$ with $p+q=n$ and $k+l=n$.

(2). If $d_n=2$ and $d_2= 2$, then the optimal graphs in the set ${\mathcal{B}}_{\pi}$ are $B(p,q)$ with $p+q=n$. In other words, for any
$G\in {\mathcal{B}}_{\pi}$,  $M_2(G)\le 4n+20$  with equality if and only if $G$ is $B(p,q)$ with $p+q=n$.

(3). If $d_n=1$ and $d_2=2$ and $s\le \frac{n-5}{2}$, then the optimal graphs in the set ${\mathcal{B}}_{\pi}$ are $B(p,q;p_1,p_2,\ldots,p_s)$ with $p_i\geq 2$ for $1\leq i\leq s$. In other words, for any $G\in {\mathcal{B}}_{\pi}$,  $M_2(G)\le 4n+2s^2+10s+20$ with equality if and only if $G$ is $B(p,q;p_1,p_2,\ldots,p_s)$ with $p_i\geq 2$ for $1\leq i\leq s$.

(4). If $d_n=1$ and $d_2=2$ and $s> \frac{n-5}{2}$, then the optimal graphs in the set ${\mathcal{B}}_{\pi}$ are $B(3,3;2,\cdots,2,1,\cdots,1)$ with $p_1=\cdots=p_{n-s-5}=2$ and $p_{n-s-4}=\cdots=p_s=1$. In other words, for any $G\in {\mathcal{B}}_{\pi}$, $M_2(G)\le sn+6n+s+10$
with equality if and only if $G$ is $B(3,3;2,\cdots,2,1,\cdots,1)$ with $p_1=\cdots=p_{n-s-5}=2$ and $p_{n-s-4}=\cdots=p_s=1$.

(5). If $d_n=1$ and $d_2\ge 3$,
 then $B_M^*(\pi)$ is an optimal graph in the set ${\mathcal{B}}_{\pi}$.
\end{theorem}

\begin{remark}
$B_M^*(\pi)$ is not the unique optimal graph for $d_n=1$ and $d_2\ge 3$. For example, let $\pi=(4^{(5)},1^{(8)})$.
Figure 2 shows two different optimal graphs.
\end{remark}

\begin{picture}(60,58)
\put(25,52){$v_{1}$}
\put(25,17){$v_{2}$}
\put(25,50){\circle{0.5}}\put(25,50){\line(0,-1){30}}\put(25,50){\line(1,-1){15}}\put(40,35){\circle{0.5}}
\put(6,35){$v_{3}$}
\put(10,35){\circle{0.5}}\put(10,35){\line(1,1){15}}\put(10,35){\line(1,-1){15}}
\put(40,37){$v_{4}$}
\put(25,20){\circle{0.5}}\put(25,20){\line(1,1){15}}
\put(50,52){$v_{5}$}
\put(25,50){\line(1,0){25}}\put(50,50){\circle{0.5}}
\put(6,20){$v_{6}$}
\put(25,20){\line(-1,0){15}}\put(10,20){\circle{0.5}}
\put(4,52){$v_{7}$}\put(14,52){$v_{8}$}
\put(10,35){\line(1,3){5}}\put(10,35){\line(-1,3){5}}\put(5,50){\circle{0.5}}\put(15,50){\circle{0.5}}
\put(40,17){$v_{9}$}\put(50,22){$v_{10}$}
\put(40,35){\line(0,-1){15}}\put(40,35){\line(1,-1){10}}\put(40,20){\circle{0.5}}\put(50,25){\circle{0.5}}
\put(65,52){$v_{11}$}\put(63,39){$v_{12}$}\put(50,32){$v_{13}$}
\put(50,50){\line(1,0){15}}\put(50,50){\line(1,-1){11}}\put(50,50){\line(0,-1){15}}
\put(65,50){\circle{0.5}}\put(61,39){\circle{0.5}}\put(50,35){\circle{0.5}}
\put(28,10){$B_M^*({\pi})$}

\put(115,52){$v_{1}$}
\put(115,17){$v_{2}$}
\put(115,50){\circle{0.5}}\put(115,50){\line(0,-1){30}}\put(115,50){\line(1,-1){15}}\put(130,35){\circle{0.5}}
\put(130,37){$v_{4}$}
\put(115,20){\circle{0.5}}\put(115,20){\line(1,1){15}}
\put(100,45.5){$v_{3}$}
\put(115,50){\line(-2,-1){15}}\put(100,42.5){\circle{0.5}}
\put(100,24.5){$v_{5}$}
\put(100,42.5){\line(0,-1){15}}\put(115,20){\line(-2,1){15}}\put(100,27.5){\circle{0.5}}
\put(135,52){$v_{6}$}
\put(115,50){\line(1,0){20}}\put(135,50){\circle{0.5}}
\put(96,20){$v_{7}$}
\put(115,20){\line(-1,0){15}}\put(100,20){\circle{0.5}}
\put(81,47.5){$v_{8}$}\put(81,37.5){$v_{9}$}
\put(100,42.5){\line(-3,1){15}}\put(100,42.5){\line(-3,-1){15}}\put(85,37.5){\circle{0.5}}\put(85,47.5){\circle{0.5}}
\put(130,17){$v_{10}$}\put(140,22){$v_{11}$}
\put(130,35){\line(0,-1){15}}\put(130,35){\line(1,-1){10}}\put(130,20){\circle{0.5}}\put(140,25){\circle{0.5}}
\put(80,32.5){$v_{12}$}\put(80,22.5){$v_{13}$}
\put(100,27.5){\line(-3,1){15}}\put(100,27.5){\line(-3,-1){15}}\put(85,22.5){\circle{0.5}}\put(85,32.5){\circle{0.5}}
\put(118,10){$G'$}
\put(60,5){\bf Figure 2}
\end{picture}

For two different non-increasing graphic sequences $\pi=(d_1,d_2,\ldots,d_n)$ and $\pi'=(d_1',d_2',\ldots,d_n')$,
we write $\pi\triangleleft\pi'$ if  $\sum_{i=1}^nd_i=\sum_{i=1}^nd_i'$
and $\sum_{i=1}^jd_i\leq \sum_{i=1}^jd_i'$ for all $j=1,2,\ldots,n$. Such an ordering is called {\it majorization} \cite{Marshall1979}.

\begin{theorem}\label{differentdegree}
Let $\pi$ and $\pi'$ be two non-increasing bicyclic degree sequences.
If  $\pi\triangleleft\pi'$, then $M_2(\pi)\le M_2(\pi')$ with equality if and only if $\pi=\pi'$.
\end{theorem}

\section{Proof of Theorem \ref{general}}
To prove the theorem, the following lemmas are needed.

\begin{lemma}\label{trans1}(\cite{Liu2012})
Let $G=(V,E)$ be a connected graph with $v_1u_1\in E, v_2u_2\in E$, $v_1v_2\notin E$ and $u_1u_2\notin E$. Let $G'=G-u_1v_1-u_2v_2+v_1v_2+u_1u_2$.
If $d(v_1)\geq d(u_2)$ and $d(v_2)\geq d(u_1)$, then $M_2(G')\geq M_2(G)$. Moreover, $M_2(G')>M_2(G)$ if and only if both two inequalities are strict.
\end{lemma}

\begin{lemma}\label{trans2}(\cite{Liu2012})
Suppose $G\in \Gamma(\pi)$, and there exist three vertices u, v, w of a connected graph G such that $uv\in E(G), uw\notin E(G), d(v)<d(w)\leq d(u)$,
and $d(u)>d(x)$ for all $x\in N(w)$. Then, there exists another connected graph $G'\in \Gamma(\pi)$ such that $M_2(G)<M_2(G')$.
\end{lemma}

\begin{lemma}\label{neighbor}(\cite{Liu2014})
For any graphic sequence $\pi$ with $n\geq 3$, there exists an optimal graph $G\in \Gamma(\pi)$
such that $\{v_2,v_3\}\subseteq N(v_1)$.
\end{lemma}

\begin{lemma}\label{i=3}
Let $\pi$ be a graphic sequence satisfying the conditions in Theorem \ref{general}.
Then there is an optimal graph $G\in\Gamma(\pi)$ such that $v_1v_2v_3$ forms a triangle.
\end{lemma}
\begin{proof}
To prove Lemma \ref{i=3}, we need to prove following claims first.

{\bf Claim 1.} There is an optimal graph $G\in \Gamma(\pi)$ such that $\{v_2,v_3\}\subseteq N(v_1)$ and there exists a cycle $C_{t_1}\subseteq G$ such that
$v_1\in C_{t_1}$.

Assume that Claim 1 does not hold for any optimal graph $G\in \Gamma(\pi)$. By Lemma \ref{neighbor}, we may suppose that G is an optimal graph in
$\Gamma(\pi)$ such that $\{v_2,v_3\}\subseteq N(v_1)$. So $v_1$ is not in any cycle of any optimal graph $G\in \Gamma(\pi)$. Since $d_1\geq c+2$,
there exists a shortest path $P=u\cdots v_1 \cdots xy$ connecting $u$ and $y$ such that $v_1$ is on the path,
where $u\in C_{t_1}$ and $d(y)=1$, $x\in P, x\in N(y)$. Suppose $w\in N(u)\bigcap V(C_{t_1})$.

If $d(w)\leq d(x)$, let $G_1=G+ux+wy-wu-xy$. By Lemma \ref{trans1} $M_2(G_1)\geq M_2(G)$. Note that $G_1\in \Gamma(\pi)$, $v_1$ is in some cycle of $G_1$
and ${v_2,v_3}\subseteq N(v_1)$, a contradiction. If $d(u)\leq d(x)$, let $G_2=G+wx+uy-wu-xy$. By Lemma \ref{trans1} $M_2(G_2)\geq M_2(G)$. For the same
reason, it's a contradiction. Thus, $min\{d(u), d(w)\}>d(x)$.

Then take $z\in (N(x)\bigcap V(P))\backslash \{y\}$. Similarly, $min\{d(u), d(w)\}>d(z)$. It can be proved that $min\{d(u), d(w)\}>d(v_1)$
by repeating this process, which is a contradiction. Thus, Claim 1 holds.

{\bf Claim 2.}  There is an optimal graph $G\in \Gamma(\pi)$ such that there exists a cycle $C_{t_1}\subseteq G$ which contains $v_1v_2$ and $v_3\in N(v_1)$.

Assume that Claim 2 does not hold for any optimal graph $G\in \Gamma(\pi)$. By Claim 1, there exists an optimal graph $G\in \Gamma(\pi)$ such that
$v_1\in V(C_{t_1})$ and $\{v_2,v_3\}\subseteq N(v_1)$. Then $v_2\notin V(C_{t_1})$, and there are two cases for $v_1$ and $v_2$.

Case 1. There is a shortest path $P=v_1v_2xy\cdots z$ connecting $v_1$ and $z$ such that $v_2$ is on the path P, where $d(z)=1$.
Choose $\{u,v\}\subseteq V(C_{t_1})$ such that $uv\in E(C_{t_1})$ and suppose $max\{d(u),d(v)\}=d(u)$. If $d(u)\geq d(x)$, let $G'=G+uv_2+vx-uv-v_2x$.
By Lemma \ref{trans1}, $M_2(G')\geq M_2(G)$ and note that $G'\in \Gamma(\pi)$ and Claim 2 holds for $G'$, a contradiction. Thus $max\{d(u),d(v)\}=d(u)<d(x)$.
Repeating the above process, we can conclude $d(u)<d(z)=1$, a contradiction. So case 1 does not hold.

Case 2. There is not any path connecting $v_1$ and $z$ such that $v_2$ is on the path, where $z$ is the arbitrary vertex in G and $d(z)=1$.
So it is obvious that $v_2$ is in another cycle $C_{t_2}$ of $G$ and $v_1\notin C_{t_2}$. Let $u_1\in N(v_1)\bigcap V(C_{t_1})$ and
$u_2\in N(v_2)\bigcap V(C_{t_2})$. By the definition of $v_1$, $v_2$, $d(v_1)\geq d(u_2)$, $d(v_2)\geq d(u_1)$. Let $G'=G-v_1u_1-v_2u_2+v_1v_2+u_1u_2$.
By Lemma \ref{trans1}, $M_2(G')\geq M_2(G)$ and note that $G'\in \Gamma(\pi)$. $v_1v_2$ is in the same cycle of $G'$, a contradiction.
So case 2 does not hold. Thus, Claim 2 holds.

{\bf Claim 3.} There is an optimal graph $G\in \Gamma(\pi)$ such that $\{v_1v_2, v_1v_3\}\subseteq E(C_{t_1})$.

By Claim 2, there is an optimal graph $G\in \Gamma(\pi)$ such that there exists a cycle $C_{t_1}\subseteq G$ which contains $v_1v_2$ and $v_3\in N(v_1)$.
If claim 3 does not hold, $v_3\notin V(C_{t_1})$, then $v_2v_3\notin E(G)$. Choose $u\in (V(C_{t_1})\bigcap N(v_2))\backslash \{v_1\}$ and
$v\in N(v_3)\backslash \{v_1\}$. If $uv\in E(G)$, let $C_{t_1}=v_1v_2uvv_3v_1$ and $\{v_1v_2, v_1v_3\}\subseteq E(C_{t_1})$, a contradiction.
So $uv\notin E(G)$. Let $G'=G+v_2v_3+uv-vv_3-uv_2$. By Lemma \ref{trans1}, $M_2(G')\geq M_2(G)$ and $G'\in \Gamma(\pi)$. Claim 3 holds for $G'$.

Thus, by Claim 3, there is an optimal graph $G\in \Gamma(\pi)$ such that
$\{v_1v_2, v_1v_3\}\subseteq E(C_{t_1})$. If $v_2v_3\notin E(G)$, choose $v\in (N(v_3)\bigcap V(C_{t_1}))\backslash \{v_1\}$. Because $d_2\geq 3$,
there are two cases for the vertices in $N(v_2)$.

Case 1. There is $u\in N(v_2)\backslash V(C_{t_1})$ such that $uv\notin E$. Let $G'=G+v_2v_3+uv-uv_2-vv_3$.
By Lemma \ref{trans1}, $M_2(G')\geq M_2(G)$ and $G'\in \Gamma(\pi)$. Since $v_1v_2v_3$ forms a triangle in $G'$, Lemma \ref{i=3} holds.

Case 2. All vertices in $N(v_2)\backslash v_1$ connect with $v$. So $d(v)\ge 3$. Then $d(v_3)\ge 3$. We can choose $u\in N(v_2)\backslash V(C_{t_1})$ and
$v'\in N(v_3)\backslash V(C_{t_1})$. Let $G'=G+v_2v_3+uv'-v_2u-v_3v'$. By Lemma \ref{trans1}, $M_2(G')\geq M_2(G)$ and $G'\in \Gamma(\pi)$.
Since $v_1v_2v_3$ forms a triangle in $G'$, Lemma \ref{i=3} is proved.
\end{proof}

\begin{lemma}\label{v1vi}
Let $\pi$ be a graphic sequence satisfying the conditions in Theorem \ref{general}.
G is an optimal graph in $\Gamma(\pi)$. If $v_1v_2v_3$, $v_1v_2v_4$, $\ldots$, $v_1v_2v_{i-1}$ form $i-3$ triangles in G, where $4\le i\le c+2$,
there is an optimal graph $G'$ (isomorphic or not isomorphic to G) in $\Gamma(\pi)$ such that
$v_1'v_2'v_3'$, $v_1'v_2'v_4'$, $\ldots$, $v_1'v_2'v_{i-1}'$ form $i-3$ triangles in $G'$ and $v_1'v_i'\in E(G')$.
\end{lemma}

\begin{proof}
If $v_1v_i\notin E(G)$, $\forall v\in N(v_1)\backslash \{v_2,\ldots,v_{i-1}\}, d(v)<d(v_i)$ otherwise we may exchange the label of $v$ and $v_i$.
Then by Lemma \ref{trans2} we may assume there exists $u\in N(v_i)$ such that $d(u)=d_1$. Suppose $u=v_j$, then $d_1=d_2=\cdots=d_j$.
There are three cases for $u=v_j$:

Case 1. $u=v_2$. The result holds after exchanging the label of $v_1$ and $v_2$.

Case 2. $u\notin \{v_2,v_3,\ldots v_{i-1}\}$, i.e. $j>i$. Then $d_1=d_2=d_i=d_j$. Let $P$ be a shortest path from $v_1$ to $v_i$.

If $\{v_2,\ldots,v_{i-1}\}\bigcap V(P)=\emptyset$, choose $x\in N(v_1)\bigcap V(P)$. Since $v_1\in N(x)\backslash v_i$ and $d_1=d_i\geq d(x)$, there must exist some vertex
$y\in N(v_i)\backslash V(P)$ such that $y\notin N(x)$. Let $G'=G+v_1v_i+xy-v_1x-v_iy$. By Lemma \ref{trans1} $M_2(G')\geq M_2(G)$. Note that
$G'\in \Gamma(\pi)$ and $v_1v_i\in E(G')$, the result holds.

If $\{v_2,\ldots,v_{i-1}\}\bigcap V(P)\neq\emptyset$, it can be proved similarly.

Case 3. $u\in \{v_2,\ldots v_{i-1}\}$, i.e. $j<i$. 
Denote set $S=N(v_1)\backslash \{v_2,\ldots,v_{i-1},N(v_j)\}$,

Case 3.1. $S\neq \emptyset$, choose $w\in S$. Note that $d(v_i)\geq d(w)$ and $d(v_1)\geq d(u)$. Let $G'=G+v_1v_i+v_jw-v_1w-v_iv_j$.
Then $M_2(G')\geq M_2(G)$ by Lemma \ref{trans1} and $G'\in \Gamma(\pi)$, $v_1v_i\in E(G')$.

Case 3.2. $S=\emptyset$. Assume $U=\{v_3,v_4,\ldots,v_{j-1},v_{j+1},\ldots,v_{i-1}\}\backslash N(v_j)$ and $|U|=l>0$. Suppose
$U=\{v_{i_1},v_{i_2},\ldots,v_{i_l}\}$. Note that $U\subseteq N(v_1)$. Since $d_1=d_j$, there exists not less than $l$ vertices in
$N(v_j)\backslash N(v_1)$. Choose $l$ vertices $u_1,u_2,\ldots,u_l$ from $N(v_j)\backslash N(v_1)$.
Let $G'=G+v_{i_1}v_j+\cdots+v_{i_l}v_j-v_{i_1}v_2-\cdots-v_{i_l}v_2+u_1v_2+\cdots+u_lv_2-u_1v_j-\cdots-u_lv_j$. It can be concluded that
$M_2(G')\geq M_2(G)$ by using Lemma \ref{trans1} $l$ times. Then relabel $v_j$ as $v_1$, $v_1$ as $v_2$ and $v_2$ as $v_j$ in $G'$. $v_1v_i\in E(G')$.
If $|U|=0$, we can do the last step directly. Hence, $v_1v_i\in E(G')$.
\end{proof}

\begin{lemma}\label{introduction}
Let $\pi$ be a graphic sequence satisfying the conditions in Theorem \ref{general}.
Then there is an optimal graph  $G\in\Gamma(\pi)$ such that $v_1v_2v_3$, $\ldots$, $v_1v_2v_{c+3}$ form $c+1$ triangles.
\end{lemma}

\begin{proof}
The lemma can be proved by induction.
For $i=3,$ the result holds by Lemma \ref{i=3}.
 Assume that for $i-1$, the assertion holds, i.e., there is an optimal graph $G\in \Gamma(\pi)$
in which $\{v_1,v_2,v_3\},\ldots,\{v_1,v_2,v_{i-1}\}$ form $i-3$ triangles.
By Lemma \ref{v1vi}, we may assume $v_1v_i\in E(G)$.
To finish the introduction, it suffices to prove the following claims. For convenience, let
$C_j$ denote triangle $v_1v_2v_j$ for $3\le j\le i-1$.

{\bf Claim 1.} There is an optimal graph $G\in \Gamma(\pi)$ in which there exists a cycle $C_{t'}$ such that $v_1\in V(C_{t'})$,
where $C_{t'}\neq C_j$ for $3\le j\le i-1$.

If Claim 1 doesn't hold for any optimal graph, $v_1\notin C$, $\forall C\neq C_j$ for $3\le j\le i-1$. Assume $C_{t'}$ is a cycle in G and
$C_{t'}\neq C_j$ for $3\le j\le i-1$. Since $d_1\geq d_2\geq c+2$ and there are $c+1$ cycles,
there exists two vertices $u,w\in V(C_{t'})$, $uw\in E(C_{t'})$ and a path $P=u\cdots v_1xy\cdots z$, where $x\notin\{v_2,v_3,\ldots,v_i\}$,
$u\in C_{t'}$ and $d(z)=1$. Note that if $x=v_{j'}$, $2\le j'\le i$, we relabel the path by
$u\cdots v_1v_{j'}xy\cdots z$ and start from $v_{j'}x$ instead of $v_1x$. Let $w\in N(u)\bigcap V(C_{t'})$.

If $d(u)\geq d(x)$, let $G_1=G+uv_1+wx-v_1x-uw$. By Lemma \ref{trans1} $M_2(G_1)\geq M_2(G)$. Note that $G_1\in \Gamma(\pi)$, $v_1$ is
in some cycle not $C_j$ of $G_1$ for $3\le j\le i-1$, a contradiction.
If $d(w)\geq d(x)$, let $G_2=G+wv_1+ux-v_1x-uw$. By Lemma \ref{trans1} $M_2(G_2)\geq M_2(G)$. For the same reason, it's a contradiction.
Thus, $max\{d(u), d(w)\}<d(x)$.

Then take $y\in (N(x)\bigcap V(P))\backslash \{v_1\}$. Similarly, $max\{d(u), d(w)\}<d(y)$. It can be proved that $max\{d(u), d(w)\}<d(z)=1$
by repeating this process, which is a contradiction. Thus, Claim 1 holds.

{\bf Claim 2.} There is an optimal graph $G\in \Gamma(\pi)$ in which there exists a cycle $C_{t'}$ such that $v_1v_2\in E(C_{t'})$,
where $C_{t'}\neq C_j$ for $3\le j\le i-1$.

If Claim 2 doesn't hold for any optimal graph, by Claim 1, we may assume there is an optimal graph $G\in \Gamma(\pi)$ in which
there exists a cycle $C_{t'}$ such that $v_1\in V(C_{t'})$ and $v_2\notin V(C_{t'})$, where $C_{t'}\neq C_j$ for $3\le j\le i-1$.
Because $v_1\in V(C_{t'})$ and there remains $c+4-i$ cycles except $C_j$, $3\le j\le i-1$ and
$d(v_2)-|N(v_2)\bigcap \{v_1,v_3,\ldots,v_{i-1}\}|=c+4-i$, there exists a vertex $z$ and a path $P=v_1v_2xy\cdots z$,
where $d(z)=1$, $x\notin \{v_3,\ldots,v_{i-1}\}$ and P is the shortest path connecting $v_1$ and $z$ such that $v_2$ is on it.
Choose $\{u,v\}\subseteq V(C_{t'})\backslash \{v_1\}$ such that $uv\in E(C_{t'})$.
Note that if  $v_1v_j\in E(C_{t'})$ for $2<j<i$, there is $C_{t''}\subseteq G$ such that $v_1v_2\in E(C_{t''})$.
So $\{v_3,\ldots,v_{i-1}\}\bigcap\{u,v\}=\emptyset$.

Suppose that $max\{d(u),d(v)\}=d(u)$. If $d(u)\geq d(x)$, let $G'=G+uv_2+vx-v_2x-uv$. By Lemma \ref{trans1} $M_2(G')\geq M_2(G)$. Note that
Claim 2 holds for $G'$ which is a contradiction. Thus, $max\{d(u),d(v)\}=d(u)<d(x)$. Similarly, $max\{d(u),d(v)\}<d(y)$. Repeating the above process,
we will yield that $max\{d(u),d(v)\}<d(z)=1$, a contradiction. Thus, Claim 2 holds.

{\bf Claim 3.} There is an optimal graph $G\in \Gamma(\pi)$ in which there exists a cycle $C_{t'}$ such that $v_1v_2,v_1v_i\in E(C_{t'})$,
where $C_{t'}\notin C_j$ for $3\le j\le i-1$.

If Claim 3 doesn't hold for any optimal graph, by Claim 2 and $v_1v_i\in E(G)$, we may assume there is an optimal graph $G\in \Gamma(\pi)$
in which there exists a cycle $C_{t'}$ such that $v_1v_2\in E(C_{t'})$ and $v_i\notin V(C_{t'})$, where $C_{t'}\neq C_j$ for $3\le j\le i-1$.
Choose $u\in (N(v_2)\bigcap V(C_{t'}))\backslash \{v_1\}$ and $v\in N(v_i)\backslash \{v_1\}$. If $u=v$, Claim 3 holds. Thus, $u\neq v$. Then,

Case 1. $u\notin \{v_3,\ldots,v_{i-1}\}$. Let $G'=G+v_2v_i+uv-v_2u-v_iv$. $M_2(G')\geq M_2(G)$ by Lemma \ref{trans1}. Note that $G'\in \Gamma(\pi)$.
So Claim 3 holds for $G'$.

Case 2. $u\in \{v_3,\ldots,v_{i-1}\}$. Choose $w\in N(u)\bigcap V(C_{t'})\backslash \{v_2\}$.

Case 2.1. $w\notin \{v_3,\ldots,v_{i-1}\}\backslash \{u\}$. Let $G'=G+uv_i+wv-uw-v_iv$. Because $d(v_i)\geq d(w)$ and $d(u)\geq d(v_i)$,
$M_2(G')\geq M_2(G)$ by Lemma \ref{trans1}. Note that $G'\in \Gamma(\pi)$. So Claim 3 holds for $G'$.

Case 2.2. $w\in \{v_4,\ldots,v_{i-1}\}\backslash \{u\}$. Let $G'=G+uv_i+wv-uw-v_iv$. By condition $(iii)$, $d(v_i)=d(w)$. So $M_2(G')\geq M_2(G)$
by Lemma \ref{trans1}. Note that $G'\in \Gamma(\pi)$. So Claim 3 holds for $G'$.

Case 2.3. $w=v_3$. Then there is another cycle $C_{t''}=v_2v_3uv_1v_2$ in $G$ such that $v_1v_2\in E(C_{t''})$ and $v_i\notin V(C_{t''})$.
Then by the same method using in Case 2.2, we can conclude that Claim 3 holds.

{\bf Claim 4.} There is an optimal graph $G\in \Gamma(\pi)$ such that $\{v_1,v_2,v_3\},\ldots,\{v_1,v_2,v_i\}$ form $i-2$ triangles in $G$.

By Claim 3, there is an optimal graph $G\in \Gamma(\pi)$ in which there exists a cycle $C_{t'}$ such that $v_1v_2,v_1v_i\in E(C_{t'})$,
where $C_{t'}\neq C_j$ for $3\le j\le i-1$. If Claim 4 doesn't hold for any optimal graph, we may assume $v_2v_i\notin E(C_{t'})$.
Choose $u\in (N(v_2)\bigcap V(C_{t'}))\backslash\{v_1\}$ and $v\in (N(v_i)\bigcap V(C_{t'}))\backslash \{v_1\}$. Note that $u$ and $v$ can be the same vertex.
There are two cases for $u$.

Case 1. $u\notin \{v_3,\ldots,v_{i-1}\}$ which implies $v\notin \{v_3,\ldots,v_{i-1}\}$.

Case 1.1. $d_3\geq 3$. Choose $w\in N(v_3)\backslash \{v_1,v_2\}$. Let $G_1=G+v_3v_i+wv-v_3w-v_iv$ and $G_2=G_1+v_2v_i+v_3u-v_3v_i-v_2u$.
By Lemma \ref{trans1} $M_2(G_2)\geq M_2(G_1)\geq M_2(G)$. Note that $G_1,G_2\in \Gamma(\pi)$ and Claim 4 holds for $G_2$.

Case 1.2. $d_3=2$. Then $d_1\geq c+3$ by condition $(iv)$. So we can choose a vertex $x\in N(v_1)\backslash \{v_2,\ldots,v_i\}$.
Let $G'=G+v_2v_i+v_1u+xv-v_2u-v_iv-v_1x$. Note that $G'\in \Gamma(\pi)$ and $G'$ is connected and $d(v_i)=d(u)=d(v)=2\geq d(x)$.
By elemental calculation, $M_2(G')-M_2(G)=(d_1-2)(2-d(x))\geq 0$. So $M_2(G')\geq M_2(G)$ and Claim 4 holds for $G'$.

Case 2. $u\in \{v_3,\ldots,v_{i-1}\}$. Since $d_2\geq c+2$, we can choose $u'\in N(v_2)\backslash V(C_{t'})\backslash \\\{v_3,\ldots,v_{i-1}\}$.
Let $G'=G+v_2v_i+u'v-v_2u'-v_iv$. By Lemma \ref{trans1} $M_2(G')\geq M_2(G)$. It is easy to check that $G'\in \Gamma(\pi)$ and Claim 4 holds for $G'$.

Thus, we can conclude by introduction that there is an optimal graph $G\in \Gamma(\pi)$
in which $\{v_1,v_2,v_3\},\{v_1,v_2,v_4\},\ldots,\{v_1,v_2,v_{c+3}\},$ form $c+1$ triangles.
\end{proof}

Now we are ready to prove Theorem \ref{general}.

\begin{proof}
The first part of the theorem have be proved by Lemma \ref{introduction}.
So we may assume $\{v_1,v_2,v_3\},\ldots,\{v_1,v_2,v_{c+3}\}$ form $c+1$ triangles in an optimal graph $G\in \Gamma(\pi)$.

Then an ordering $\prec$ of $V(G)$ can be created by the
breadth-first search as follows: firstly, let $v_1\prec v_2\prec\cdots\prec v_{c+3}$; secondly, append all neighbors $u_{c+4},\ldots,u_{d_1+1}$
of $N(v_1)\backslash \{v_2,\ldots,v_{c+3}\}$ to the order list,
these neighbors are ordered such that $u\prec v$ whenever $d(u)>d(v)$ (in the remaining case
the ordering can be arbitrary); thirdly, append all neighbors $u_{d_1+2},u_{d_1+3},\ldots,u_{d_1+d_2-2}$ of $N(v_2)\backslash \{v_1,v_3,\ldots,v_{c+3}\}$
to the ordered list, these neighbors are ordered such that $u\prec v$ whenever $d(u)>d(v)$ (in the remaining case the ordering can be arbitrary);
with the same method we can append the vertices of $N(v_3)\backslash \{v_1,v_2\}, \cdots, N(v_{c+3})\backslash \{v_1,v_2\}$ to the ordered list.
Then, append the vertices $N(x)\backslash \{v_1\}$ to the ordered list, where $d(x)=max\{d(y):y\in N(v_1)\backslash \{v_2,v_3,\ldots,v_{c+3}\}\}$.
Repeat the last process recursively with all vertices $v_1,v_2,\ldots$, until all vertices of G are processed.

Then $H_0=\{v_1\}$.
By the construction of $\prec$, $u\prec v$ implies $h(u)\leq h(v)$. For $v\in H_i(G),i>0$,
we call the unique vertex $u\in N(v)\bigcap H_{i-1}(G)$ the parent of v.
So $u\prec v$, if $u$ is the parent of $v$. Moreover, because the vertices are appended to
the ordered list recursively, if there are two edges $uu_1\in E(T)$ and $vv_1\in E(T)$ such that $u\prec v$, $h(u)=h(u_1)+1$ and
$h(u)=h(v_1)+1$, then $u_1\prec v_1$. 

To prove  the assertion,  it suffices to show that $d(u)\geq d(v)$ holds for each two vertices $u,v\in V(G)$ and $u\prec v$.

If the above proposition doesn't hold, assume $v_i$ is the first vertex in the ordering of $\prec$ with the property $v_i\prec u$
and $d(v_i)<d(u)$ for some $u\in V(G)$. Clearly, $v_i\notin \{v_1,v_2,v_3,\ldots,v_{c+3}\}$ and if $v\prec v_i$, $d(v)\geq d(u)$ holds
for each $u$ with $v\prec u$. Suppose $v_j$ is the first vertex in the ordering $\prec$ such that $v_i\prec v_j$ and
$d(v_j)=max\{d(v_t):i+1\leq t\leq n\}$. By the choice of $v_i$, we can conclude that $v_i\prec v_j$, but $d(v_i)<d(v_j)$.
Let $w_i$ and $w_j$ be the parents of $v_i$ and $v_j$, respectively. Note that $d(v_i)<d(v_j)$. Then $w_i\neq w_j$ and $w_i\prec w_j$
by the construction of $\prec$. It is obvious that $w_iv_j\notin E(G)$. Otherwise there is a cycle in $G$ such that $w_i, w_j, v_j$ are on it
and $E(G)\geq n+c+1$ because $w_i\prec w_j $ and $v_j\notin \{v_1,v_2,v_3,\ldots,v_{c+3}\}$.
Let's consider the following two cases.

Case 1. $w_iv_i$ is in the shortest path that connects $w_j$ and $v_1$. 
We can conclude that $w_i\prec v_i\prec w_j\prec v_j$ and $d(w_i)>d(v_j)>d(w_j)$ by the definition of $v_i$ and $v_j$.
Now we shall prove the following Claim.

{\bf Claim}. There exists some $y\in N(v_j)\backslash \{w_j\}$ such that $d(w_i)=d(v_j)=d(y)$ and $v_iy\notin E(G)$.

Because $v_i\notin \{v_1,v_2,v_3,\ldots,v_{c+3}\}$, $v_iy\notin E(G)$ holds for every $y\in N(v_j)\backslash \{w_j\}$ for the same reason of
$w_iv_j\notin E(G)$. If $d(w_i)>d(y)$ holds for every $y\in N(v_j)\backslash \{w_j\}$, $d(w_i)>d(y)$ holds for all $y\in N(v_j)$
because $d(w_i)>d(w_j)$. So $d(w_i)\geq d(v_j)>d(v_i)$, $w_iv_i\in E(G)$ and $w_iv_j\notin E(G)$. By Lemma \ref{trans2}, there exists
another graph $G'\in \Gamma(\pi)$ such that $M_2(G)<M_2(G')$, a contradiction. Thus, there exists some $y\in N(v_j)\backslash \{w_j\}$
such that $d(w_i)\leq d(y)$. On the other hand, by $w_i\prec v_i\prec w_j\prec v_j\prec y$ and the choice of $v_j$, we have
$d(w_i)\geq d(v_j)\geq d(y)$. Hence, claim holds.

Then there exists some $y\in N(v_j)\backslash \{w_j\}$ such that $d(w_i)=d(v_j)=d(y)>d(v_i)$. Let $G_1=G+w_iv_j+v_iy-w_iv_i-v_jy$. Clearly,
$G_1\in \Gamma(\pi)$. By Lemma \ref{trans1}, $M_2(G)\leq M_2(G_1)$.

Case 2.  $w_iv_i$ is not in the shortest path that connects $w_j$ and $v_1$.

Then $w_jv_i\notin E(G)$. Otherwise we can find a cycle in $G$ such that $v_1,v_i,w_j$ or $v_2,v_i,w_j$  are on it and $E(G)\geq n+c+1$, a contradiction.
Let $G_1=G+w_iv_j+w_jv_i-w_iv_i-w_jv_j$. Then $G_1\in \Gamma(\pi)$. Because $w_i\prec v_i$ and $w_i\prec w_j$, $d(w_i)\geq d(w_j)$ by the choice of $v_i$.
By Lemma \ref{trans1}, $M_2(G)\leq M_2(G_1)$.

Note that $\{v_1,v_2,v_3\},\cdots,\{v_1,v_2,v_{c+3}\}$ still form $c+1$ triangles in $G_1$.
After getting a new graph $G_1\in \Gamma(\pi)$ such that $M_2(G)\leq M_2(G_1)$ in the above two cases, we redefine the ordering $\prec$ to $V(G_1)$
as follows: Let $v_1\prec v_2 \prec \cdots \prec v_{i-1} \prec v_j$ be the first i vertices. Then, append the rest vertices by the same
method which is used in the construction of $\prec$ of $V(G)$. In the redefined ordering, if $v\prec v_j$ or $v=v_j$, $d(v)\geq d(u)$ holds for all $v\prec u$.
Moreover, by the construction of the redefined $\prec$, if there are two edges $uu_1\in E(T)$ and $vv_1\in E(T)$ such that $u\prec v$, $h(u)=h(u_1)+1$ and
$h(u)=h(v_1)+1$, then $u_1\prec v_1$. We can also conclude $h(u)\leq h(v)$ if $u\prec v$.

So repeating the above process at most $t(t\leq n-c-3)$ times, we can get an optimal graph $G_t\in \Gamma(\pi)$ such that $d(u)\geq d(v)$ holds for
each two vertices $u,v\in V(G)$ and $u\prec v$. $G_t$ is isomorphic to the graph constructed in the theorem.
\end{proof}

\section{Proof of Theorem~\ref{main}}

\begin{lemma}\label{main1}
Let $\pi=(d_1, \cdots, d_n)$ be a bicyclic graphic degree sequence.

(1). If $d_n=2$ and $d_2\ge 3$, then the optimal bicyclic graphs in the set ${\mathcal{B}}_{\pi}$ are $B(p,1,q)$ and $B(P_k,P_l,P_1)$ with $p+q=n$ and $k+l=n$.

(2). If $d_n=2$ and $d_2=2$, then the optimal bicyclic graphs in the set ${\mathcal{B}}_{\pi}$ are $B(p,q)$ with $p+q=n$.
\end{lemma}
\begin{proof}
 If $d_n=2$ and $d_2\ge 3$,  then the only possible degree sequence is  $\pi= (3, 3, 2^{(n-2)})$ and $G$ is  $B(p,r,q)$ or $B(P_k,P_l,P_m)$.
 It is easy to see that
$M_2(B(p,1,q))=M_2(B(P_k,P_l,P_1))=4n+17
>M_2(B(p,r,q))=M_2(B(P_k,P_l,P_m))=4n+16$ for $r>1,m>1$. Hence (1) holds.

If  $d_n=2$ and $d_2= 2$, then $G$ is $B(p,q)$ with $p+q-1=n$. It is easy to see that
$M_2(B(p, q))=4n+20$. Hence (2) holds.
\end{proof}

\begin{lemma}\label{main2}
Let $\pi=(d_1, \cdots, d_n)$ be a bicyclic graphic  sequence. Suppose the number of leaves in the graph of ${\mathcal{B}}_{\pi}$ is $s$. If $d_n=1$ and $d_2=2$, then the optimal bicyclic graphs in the set ${\mathcal{B}}_{\pi}$ are

(1). $B(p,q;p_1,p_2,\ldots,p_s)$ with $p_i\geq 2$ for $1\leq i\leq s$ when $s\le\frac{n-5}{2}$.

(2). $B(3,3;2,\cdots,2,1,\cdots,1)$ with $p_1=\cdots=p_{n-s-5}=2$ and $p_{n-s-4}=\cdots=p_s=1$ when $s>\frac{n-5}{2}$.
\end{lemma}

\begin{proof}
We may write $\pi=(d_1,2^{(k)},1^{(s)})$, where $k=n-s-1$ and $d_1=2n-2k-s+2$. The lemma can be proved easily by exhaustion.

(1). $s\le\frac{n-5}{2}$ i.e $k\geq s+4$ i.e. $n\leq 2k-3$.

The optimal graphs are
$B(3,3;k-s-2,2,2,\ldots,2),B(3,3;k-s-3,3,2,\ldots,2),\\
\cdots ,B(p,q;p_1,p_2,\ldots,p_s)$ whose second Zagreb indices are all equal to
$2\times (n-k+3)(n-k+3)+2\times 2 \times (2k-n-1)+2\times 1 \times (n-k-1)=2n^2-4nk+2k^2+10n-6k+12=4n+2s^2+10s+20$,
where $p_i\geq 2$ for $1\leq i\leq s$.

(2). $s>\frac{n-5}{2}$ i.e. $4\leq k<s+4$ i.e. $n>2k-3$.

The unique optimal graph of this case is $B(3,3;2,\cdots,2,1,\cdots,1)$ whose second Zagreb index is
$2\times (n+3-k)k+1\times (n+3-k)(n-2k+3)+2\times 2 \times 2+1\times 2 \times (k-4)=n^2-nk+6n-k+9=sn+6n+s+10$,
where $p_1=\cdots=p_{k-4}=2$ and $p_{k-3}=\cdots=p_s=1$.
\end{proof}

Now we are ready to prove Theorem~\ref{main}.

\begin{proof}
 It is easy to see that the assertion  follows from Lemmas \ref{main1}, \ref{main2} and  Theorem \ref{general}.
\end{proof}

\section{Proof of Theorem~\ref{differentdegree}}
In order to prove Theorem~\ref{differentdegree},  we need some lemmas
\begin{lemma}\label{majorization}(\cite{Marshall1979})
Let $\pi$ and $\pi'$ be two different non-increasing graphic sequences. If $\pi\triangleleft\pi'$, then there exists a series of non-increasing graphic sequences
$\pi_1,\pi_2,\ldots,\pi_k$ such that $\pi=\pi_0\triangleleft\pi_1\triangleleft\pi_2\triangleleft\ldots\triangleleft\pi_k\triangleleft\pi_{k+1}\triangleleft\pi'$, where $\pi_i$ and $\pi_{i+1}$ differ only in two positions and the differences are 1 for $0\leq i\leq k$.
\end{lemma}

\begin{lemma}\label{trans3}(\cite{Liu2012})
Let $u,v$ be two vertices of a connected graph G, and $w_1,w_2,\ldots,w_k$ $(1\leq k\leq d(v))$ be some vertices of $N(v)\backslash (N(u)\bigcup \{u\})$.
Let $G'=G+w_1u+w_2u+\cdots+w_ku-w_1v-w_2v-\cdots-w_kv$. If $d(u)\geq d(v)$ and $\sum_{y\in N(u)}d(y)\geq \sum_{x\in N(v)}d(x)$, then $M_2(G')>M_2(G)$.
\end{lemma}

\begin{lemma}\label{d2=2}
Let $\pi=(d_1, \ldots, d_n)$  and $\pi'=(d_1', \ldots, d_n')$ be two bicyclic graphic degree sequence. Suppose that at most one following condition holds.

(i)\quad$d_2=3$ and $d_n=1$.

(ii)\quad$d_2'=3$ and $d_n'=1$.

If  there exist $1\le p<q\le n$  with $d_p=d_p'+1$, $d_q=d_q-1$  for $1\le p<q\le n$  and $d_i=d_i'$ for all $i\neq p, q$, then $M_2({\pi})<M_2({\pi'})$.
\end{lemma}
\begin{proof}
This Lemma can be proved by exhaustion. Let $G_\pi$ be an optimal graph with degree sequence $\pi$.
Then for each degree sequences $\pi$, the method to prove the lemma is to find all possible degree sequences $\pi_1,\pi_2$ such that $\pi_1\triangleleft\pi\triangleleft\pi_2$, where $\pi_1,\pi$ and $\pi,\pi_2$ differ only in two positions, where the difference are 1. After that, prove $M_2(G_{\pi_1})<M_2(G_{\pi})<M_2(G_{\pi_2})$. Without loss of generality, we may assume condition (i) doesn't hold.
There are four cases for $\pi$.

Case 1. $\pi=(3,3,2^{(n-2)})$.

It is easy to check that for any other bicyclic sequences $\pi'$ satisfying the conditions in Lemma \ref{d2=2}, $\pi\triangleleft\pi'$ holds and
$M_2(\pi)=4n+17<M_2(\pi')$.

Case 2. $\pi=(4,2^{(n-1)})$.

The all possible sequences for $\pi_1$ and $\pi_2$ are $\pi_1=(3,3,2^{(n-2)})$ and $\pi_2=(5,2^{(n-2)},1)$,
$\pi_2'=(4,3,2^{(n-3)},1)$. By the preceding proof and calculation, $M_2(G_{\pi_1})=4n+17$, $M_2(G_{\pi})=4n+20$,
$M_2(G_{\pi_2})=2n^2-4nk+2k^2+10n-6k+12=4n+32~(k=n-2)$, $M_2(G_{\pi_2'})=M_2(B_M^*(\pi_2'))=4n+26$ and
$M_2(G_{\pi_1})<M_2(G_{\pi})<M_2(G_{\pi_2'})<M_2(G_{\pi_2})$. Lemma \ref{d2=2} holds for this case.

Case 3. $\pi=(d_1,2^{(k)},1^{(s)})$,where $k\geq s+4$ i.e. $n\leq 2k-3$.

The all possible sequences for $\pi_1$ and $\pi_2$ are
$\pi_1=(d_1-1,2^{(k+1)},1^{(s-1)})$, $\pi_1'=(d_1-1,3,2^{(k-1)},1^{(s)})$ and $\pi_2=(d_1+1,2^{(k-1)},1^{(s+1)})$, $\pi_2'=(d_1,3,2^{(k-2)},1^{(s+1)})$.
By the preceding proof and calculation,

$M_2(G_{\pi_1})=2n^2-4nk+2k^2+6n-2k+8$,

$M_2(G_{\pi_1}')=2n^2-4nk+2k^2+7n-3k+12$,

$M_2(G_{\pi})=2n^2-4nk+2k^2+10n-6k+12$,


$M_2(G_{\pi_2})=2n^2-4nk+2k^2+14n-10k+20~for~n\leq2k-5$;

$M_2(G_{\pi_2})=n^2-n(k-1)+6n-(k-1)+9=n^2-nk+7n-k+10~for~n=2k-4,2k-3$,

$M_2(G_{\pi_2'})=2n^2-4nk+2k^2+11n-7k+17~for~n\leq 2k-4$;

$M_2(G_{\pi_2'})=2n^2-4nk+2k^2+14n-14k+28~for~n=2k-3$.

So $M_2(G_{\pi_1})<M_2(G_{\pi_1'})<M_2(G_{\pi})<M_2(G_{\pi_2'}<M_2(G_{\pi_2})$ and Lemma \ref{d2=2} holds for this case.

Case 4. $\pi=(d_1,2^{(k)},1^{(s)})$, where $4\leq k<s+4$ i.e. $n>2k-3$.

The all possible sequences for $\pi_1$ and $\pi_2$ are the same as the above case
except that the $M_2$ is different. By the preceding proof and calculation, $M_2(G_{\pi_1})=n^2-nk+5n-k+8~for~n>2k-1$;

$M_2(G_{\pi_1})=2n^2-4nk+2k^2+6n-2k+8~for~n=2k-2,2k-1$,

$M_2(G_{\pi_1'})=n^2-nk+5n-k+12$,

$M_2(G_{\pi})=n^2-nk+6n-k+9$,

$M_2(G_{\pi_2})=n^2-nk+7n-k+10$,

$M_2(G_{\pi_2'})=n^2-nk+6n-k+13$,

Lemma \ref{d2=2} also holds for this case.
\end{proof}

\begin{lemma}\label{d2>2}
Let $\pi=(d_1, \ldots, d_n)$  and $\pi'=(d_1', \ldots, d_n')$ be two bicyclic graphic degree sequence with $M_2({\pi})$ and $M_2({\pi'})$ being the maximum second Zagreb index in the set ${\mathcal{B}}_{\pi}$. Suppose that $d_2\ge 3$, $d_2'\ge 3$ and  $d_n=1$, $d_n'=1$. If  there exist $1\le p<q\le n$  with $d_p=d_p'+1$, $d_q=d_q-1$  for $1\le p<q\le n$  and $d_i=d_i'$ for all $i\neq p, q$, then $M_2({\pi})<M_2({\pi'})$.
\end{lemma}
\begin{proof}
By Theorem \ref{main}, $M_2({\pi})=M_2(B_M^*({\pi}))$. So it suffice to show that $M_2(B_M^*({\pi}))<M_2(\pi')$.

We have $v_p\prec v_q$ in the ordering of $V(B_M^*({\pi}))$ since $p<q$ and hence $d(v_p)\geq d(v_q)$.
By the proof of the last part of Theorem \ref{general},
we have $\sum_{x\in N_{B_M^*({\pi})}(v_p)}d(x)\geq \sum_{y\in N_{B_M^*({\pi})}(v_q)}d(y)$.
Let $P$ be the (one of) shortest path from $v_p$ to $v_q$ in $B_M^*({\pi})$.

If $q=2$, then $d_q\geq 4$ because $d_q'=d_q-1\ge3$. If $3\leq q\leq 4$, then $d_q\geq 3$ because $d_q'=d_q-1=2$. If $q\geq 4$, then $d_q\geq 2$. In all these
cases, there exists a vertex $v_k(k>q)$ such that $v_k\in N_{B_M^*({\pi})}(v_q)\backslash N_{B_M^*({\pi})}(v_p)$ and $v_k\notin V(P)$. Let $G=B_M^*({\pi})+v_pv_k-v_qv_k$.

Note that $G\in \Gamma(\pi')$ and $d(v_p)\geq d(v_q)$. By Lemma \ref{trans3}, $M_2(B_M^*({\pi}))<M_2(G)\leq M_2(\pi')$.
\end{proof}

Now we are ready to prove Theorem \ref{differentdegree}.

\begin{proof}
Set $\pi=(d_1,d_2,\ldots,d_n)$ and $\pi'=(d_1',d_2',\ldots,d_n')$. Since $\pi\triangleleft\pi'$, by Lemma \ref{majorization} we may suppose that $\pi$
and $\pi'$ differ only in two positions, where the difference are 1. So we may assume that $d_i=d_i'$ for $i\neq p,q$, and $d_p+1=d_p', d_q-1=d_q'$,
$1\leq p<q\leq n$.

The remaining parts of proofs follow from Lemmas~\ref{d2=2} and \ref{d2>2}.
\end{proof}




\begin{thebibliography}{}

\bibitem{Balaban1983} A.~T.~Balaban, I.~Motoc, D.~Bonchev and O.~Mekenyan, Topological indices for structure-activity correlations, {\it Topics Curr. Chem.}, 114(1983) 21-55.

\bibitem{das2014}K.~Ch.~Das, H.~U.~Jeon  and N.~Trinajsti\'{c}, Comparison between the Wiener index and the Zagreb indices and the eccentric connectivity index for trees. {\it Discrete Appl. Math.}, 71(2014) 35-41.

\bibitem{estes2014}J.~Estes and B.~Wei, Sharp bounds of the Zagreb indices of $k-$trees. {\it J. Comb. Optim.},  27(2014) 271-291.
\bibitem{Gutman2004} I.~Gutman and K.~C.~Das, The first Zagreb index 30 years after, {\it MATCH Commun. Math. Comput. Chem.}, 50(2004) 83-92.


\bibitem{gutman1972}I.~Gutman and N.~Trinajsti\'{c}, Graph theory and molecular orbitals. Total $\pi-$electron
energy of alternant hydrocarbons, {\it Chem. Phys. Lett.},  17(1972) 535-538.



\bibitem{Gutman1975} I.~Gutman, B.~Ru\v{s}\v{c}i\'{c} and C.~F.~Wilcox, Graph theory and molecular orbitals.12.Acyclic polyenes, {\it J. Chem. Phys.}, 62(1975) 3399-3405.

\bibitem{ilic2012}A.~Ili\'{c} and B.~Zhou,  On reformulated Zagreb indices, {\it  Discrete Appl. Math.,}  160 (2012) 204-209.

\bibitem{hua2013} H.~B.~Hua and  K.~Ch.~Das, The relationship between the eccentric connectivity index and Zagreb indices, {\it  Discrete Appl. Math.,}  161 (2011) 2480-2491.


\bibitem{Kier1976} L.~B.~Kier and L.~H.~Hall, {\it Molecular Connectivity in Chemistry and Drug Research}, Academic Press, San Francisco, 1976.

\bibitem{Kier1986} L.~B.~Kier and L.~H.~Hall, {\it Molecular Connectivity in Structure-Activity Analysis}, Wiley, New York, 1986.

\bibitem{Liu2012} M.~H.~Liu, B.~L.~Liu, The second Zagreb indices and Wiener polarity indices of trees with given degree sequences, {\it MATCH Commun. Math. Comput. Chem.}, 67(2012) 439-450.

\bibitem{Liu2014} M.~H.~Liu, B.~L.~Liu, The second Zagreb indices of unicyclic graphs with given degree sequences, {\it Discrete Appl. Math.}, 167(2014) 217-221.

\bibitem{Marshall1979} A.~W.~Marshall, I.~Olkin, {\it Inequalities: Theory of Majorization and Its Applications}, Academic Press, New York, 1979.

\bibitem{Nikolic2003} S.~Nikoli\'{c}, G.~Kova\v{c}evi\'{c}, A.~Mili\v{c}evi\'{c} and N.~Trinajsti\'{c}, The Zagreb indices 30 years after, {\it Croat. Chem. Acta},  76 (2003) 113-124.

\bibitem{Todeschini2000} R.~Todeschini and V.~Consonni, {\it Handbook of Molecular Descriptors}, Wiley VCH, Weinheim, 2000.


\bibitem{Zhang2007} X.-D.~Zhang, The Laplacian spectral radii of trees with degree sequences, {\it Discrete Math.}, 308(2008) 3143-3150.
    \bibitem{zhang2009}X.-D.~Zhang, The signless Laplacian spectral radius of graphs with given degree sequences, {\it  Discrete Appl. Math.}, 157(2009) 2928-2937.
\end{thebibliography}
\end{document}